\documentclass{svproc}
\usepackage{url}

\usepackage{amsfonts}
\usepackage{amsmath}
\usepackage{latexsym}
\usepackage{amssymb}
\usepackage{tikz}
\usepackage{soul}
\usetikzlibrary{arrows.meta}

\usepackage{graphicx}
\usepackage{subcaption}
\usepackage{float}
\usepackage{comment}

\begin{document}
\mainmatter 

\title{ (2,3)-Cordial Oriented Hypercubes}
\titlerunning{(2,3)-Cordial Oriented Hypercubes}

\author{Jonathan M. Mousley \\ Department of Mathematics and Statistics \\\email{jonathanmousley@gmail.com}\\ \vspace{1em} LeRoy B. Beasley\\ Department of Mathematics and Statistics \\\email{leroy.b.beasley@aggiemail.usu.edu}\\ \vspace{1em} Manuel A. Santana\\ Department of Mathematics and Statistics\\\email{manuelarturosantana@gmail.com}\\  \vspace{1em}\vspace{1em} David E. Brown \\ Department of Mathematics and Statistics\\ \email{david.e.brown@usu.edu}}
\authorrunning{Jonathan M. Mousley, et. al}
\institute{Utah State University. Logan UT 84322, USA,\\}

\maketitle

\begin{abstract}  In this article we investigate the existence of $(2,3)$-cordial labelings of oriented hypercubes. In this investigation, we determine that there exists a $(2,3)$-cordial oriented hypercube for any dimension divisible by $3$. Next, we provide examples of $(2,3)$-cordial oriented hypercubes of dimension not divisible by $3$ and state a conjecture on existence for dimension $3k + 1$. We close by presenting the only 3D oriented hypercubes up to isomorphism that are not $(2,3)$-cordial. \end{abstract}

\def \dd{\hfill \baseb \vskip .5cm}
\def \d{{\noindent \it Proof. } }
\def\ex{\begin{example}}
\def\eex{\end{example}}
\def\exx{\end{example}}
\def\t{\begin{theorem}}
\def\tt{\end{theorem}}
\def\D{\begin{definition}}
\def\DD{\end{definition}}
\def\l{\begin{lemma}}
\def\ll{\end{lemma}}
\def\c{\begin{corollary}}
\def\cc{\end{corollary}}
\def\cj{\begin{conjecture}}
\def\cjj{\end{conjecture}}
\def\e{\begin{equation}}
\def\ee{\end{equation}}
\def\p{\begin{proposition}}
\def\pp{\end{proposition}}
\def\dn{{{\cal D}_n}}

\renewcommand{\vec}[1]{\overrightarrow{#1}}

\def\myedgestyle{-{Latex[length=2mm]}}

\def\N{\mathbb{N}}
\def\Z{\mathbb{Z}}

\font\cy=wncyr10

\def\bp{\mathbf p}
\newcommand{\baseb}{\hfill \rule{2mm}{2mm}}

\section{Introduction.}
Let $G = (V,E)$ be an undirected graph with vertex set $V$ and edge set $E$, a convention we will use throughout this paper.  A $(0,1)$-labeling of the vertex set is a mapping $f:V\to \{0,1\}$  and is  said to be {\em friendly} if approximately one half of the vertices are labeled 0 and the others labeled 1. An induced labeling of the edge set  is a mapping $g:E\to \{0,1\}$ where for an edge $uv, g(uv)= \hat{g}(f(u),f(v))$ for some $\hat{g}:\{0,1\}\times\{0,1\}\to \{0,1\}$ and is said to be cordial if $f$ is friendly and about one half the edges of $G$ are labeled 0.  A graph, $G$,  is called {\em cordial} if there exists a cordial induced labeling of the edge set of $G$ \cite{C}.

In this article we investigate a labeling of directed graphs that is not simply a cordial labeling of the underlying undirected graph. The labeling scheme we investigate here was introduced by L.B. Beasley in \cite{B}. Let $D = (V,A)$ be a directed graph with vertex set $V$ and arc set $A$ with a friendly vertex set mapping $f$. Let  $g:A\to \{-1,0,1\}$ be the induced  labeling of the arcs of $D$  such that for any arc initiating at $u$ and terminating at $v$, $\overrightarrow{uv}$, $g(\overrightarrow{uv}) = f(v) - f(u)$. $D$ is said to be $(2,3)$-cordial if there exists a friendly vertex set mapping $f$ such that $g$ labels approximately one third of arcs $0$, approximately one third of arcs $1$, and approximately one third of arcs $-1$. Applications of balanced graph labelings can be found in the introduction of \cite{D}. 

In \cite{B3}, $(2,3)$-cordial labelings are investigated on oriented trees, oriented paths, orientations of the Petersen graph, and complete graphs. In this article we consider $(2,3)$-cordial labelings on oriented hypercubes. We confirm the existence of $(2,3)$-cordial oriented hypercubes for every dimension $3k$ for $k \in \N$. Additionally, we provide examples of $(2,3)$-cordial oriented hypercubes for dimensions $4$ and $7$ and conjecture that there exists a $(2,3)$-oriented hypercube of dimension $3k + 1$ for every $k \in \mathbb{N}$. We close by presenting the only 3D oriented hypercubes up to isomorphism that are not $(2,3)$-cordial, that is we present the only two 3D oriented hypercube up to isomorphism that do not admit a $(2,3)$-cordial labeling. 

\section{Preliminaries.}
\begin{definition}
 Let $Z$ be a finite set and $f:Z\to\{0,1\}$ be a mapping.  The mapping $f$ is a called a $(0,1)$-labeling of the set $Z$.  If $-1\leq |f^{-1}(0)| - |f^{-1}(1)|\leq 1$, that is, the number of elements of $Z$ labeled 0 and the number of elements of $Z$ labeled 1 are about equal, then we say that the labeling $f$ is \underline{\em friendly}.
\end{definition}

Let $G=(V,E)$ be an undirected graph with vertex set $V$ and edge set $E$.  Let  $f:V\to \{0,1\}$  be a labeling of $V$.   An induced labeling of the edge set  is a mapping $g:E\to \{0,1\}$ where for an edge $uv, g(uv)= \hat{g}(f(u),f(v))$ for some $\hat{g}:\{0,1\}\times\{0,1\}\to \{0,1\}$ and is said to be cordial if $f$ and $g$ are both  friendly labelings.  A graph $G$ is {\em cordial} if there exists a cordial induced labeling of the edge set of $G$.  In this article, as in \cite{B}, we define a cordial labeling of directed graphs that is not  simply a cordial labeling of the underlying undirected graph.  

\begin{definition} Let $D=(V,A)$ be a directed graph with vertex set $V$ and arc set $A$.  Let $f:V\to \{0,1\}$ be a friendly vertex labeling and let $g$ be the  induced labeling of the arc set,  $g:A\to \{0,1,-1\}$ where for an arc $\overrightarrow{uv}, g(\overrightarrow{uv})=f(v)-f(u)$.  The labelings $f$ and $g$  are  {\em $(2,3)$-cordial} if $f$ is friendly and about one third the arcs of $D$ are labeled 1, one third are labeled -1 and one third labeled 0, that is, for any $i,j\in\{0,1,-1\},$  $-1\leq |g^{-1}(i)| - |g^{-1}(j)|\leq 1 $.  
   A digraph, $D$,  is called {\em $(2,3)$-cordial} if there exists  $(2,3)$-cordial  labelings $f$ of the vertex set and $g$  of the arc set of $D$.   An undirected graph $G$ is said to be \underline{\em $(2,3)$-orientable} if there is an orientation of the edges of $G$ which is a $(2,3)$-cordial digraph.
 \label{def_induced_label}
\end{definition}
See \cite{B3} for an equivalent definition of $(2,3)$-cordiality and $(2,3)$-orientability beginning from the view of quasi-groups and \textit{quasi-group cordiality} introduced in \cite{P}.
\begin{definition}
Let $\mathcal{D}_n$ be the set of all digraphs on $n$ vertices. We will define $\mathcal{T}_n$ as the subset of $\mathcal{D}_n$ that consists of all digon-free digraphs, where a digon is a $2$ cycle on a digraph.
\end{definition}
\begin{definition} Let $D=(V,A)$ be a digraph with vertex labeling $f:V\to \{0,1\}$ and with induced arc labeling $g:A\to\{0,1,-1\}$.  Define $\Lambda_{f,g} : \mathcal{D}_n\to\mathbb{N}^3$ by $\Lambda_{f,g}(D)=(\alpha,\beta,\gamma)$ where $\alpha=|g^{-1}(1)|, \beta=|g^{-1}(-1)|,$ and $\gamma =|g^{-1}(0)|$.  \end{definition}

Let $D\in\mathcal{T}_n$ and let $D^R$ be the digraph such that every arc of $D$ is reversed, so that $\overrightarrow{uv}$ is an arc in $D^R$ if and only if $\overrightarrow{vu}$ is an arc in $D$.  Let $f$ be a $(0,1)$-labeling of the vertices of $D$ and let $g(\overrightarrow{uv}) =f(v)-f(u)$ so that $g$ is a $(1,-1,0)$-labeling of the arcs of $D$.  Let $\overline{f}$ be the complementary $(0,1)$-labeling of the vertices of $D$, so that $\overline{f}(v)=0$ if and only if $f(v)=1$.  Let $\overline{g}$ be the corresponding induced arc labeling of $D$, $\overline{g}(\overrightarrow{uv}) =\overline{f}(v)-\overline{f}(u)$.

\begin{lemma} Let $D\in\mathcal{T}_n$ with vertex labeling $f$ and induced arc labeling $g$.  Let $\Lambda_{f,g}(D)=(\alpha,\beta,\gamma)$.   Then \begin{enumerate}\item $\Lambda_{f,g}(D^R)=(\beta,\alpha,\gamma)$. \item $\Lambda_{\overline{f},\overline{g}}(D)=(\beta,\alpha,\gamma)$, and \item $\Lambda_{\overline{f},\overline{g}}(D^R)=\Lambda_{f,g}(D).$\end{enumerate} 
\label{lambda_lemma}\end{lemma}
\begin{proof} If an arc is labeled 1, -1, 0 respectively then reversing the labeling of the incident vertices gives a labeling of -1, 1, 0 respectively,  If an arc $\overrightarrow{uv}$ is labeled 1, -1, 0 respectively, then $\overrightarrow{vu}$ would be labeled -1, 1, 0 respectively. \end{proof}
 
\begin{example}Now, consider a graph, {\cy X}$_n$ in $\mathcal{G}_n$   consisting of three parallel edges and n-6 isolated vertices.  Is  {\cy X}$_n$  $(2,3)$-orientable?  If $n=6$, the answer is no, since any friendly labeling of the six vertices would have either  no arcs labeled 0 or two arcs labeled 0.  In either case, the orientation would never be $(2,3)$-cordial. That is  {\cy X}$_6$ is not $(2,3)$-orientable, however  with additional vertices like  {\cy X}$_7$ the graph is $(2,3)$-orientable.\end{example}

Thus, for our investigation here, we will use the convention that a graph, $G$, is $(2,3)$-orientable/$(2,3)$-cordial  if and only if  the subgraph of $G$ induced by its non-isolated vertices, $\hat{G}$, is $(2,3)$-orientable/$(2,3)$-cordial.

\section{Existence.} 
We begin with examples of (2,3)-cordial oriented hypercubes for dimensions less than and equal to $3$.

\begin{example}[Dimension $1$] Given in Figure \ref{exist_d1} is a $1$-dimensional oriented hypercube $C_1$ that is (2,3)-cordial as by the friendly vertex labeling $f$ shown, $\Lambda_{f,g}(C_1) = (1,0,0)$.\label{example_d1}\end{example}

\begin{example}[Dimension $2$] Given in Figure \ref{exist_d2} is a $2$-dimensional oriented hypercube $C_2$ that is (2,3)-cordial as by the friendly vertex labeling $f$ shown, $\Lambda_{f,g}(C_2) = (1,1,2)$.\label{example_d2}\end{example}

%Figure: Existence of Dim's 1 and 2
\begin{figure}
\centering
\begin{subfigure}[b]{0.4\textwidth}
\centering
\begin{tikzpicture}[scale=1.50]
\tikzset{vertex/.style = {shape=circle, draw, minimum size=1em}}

%Dimension 1

\node[vertex] (a) at (0,1) {0};
\node[vertex] (b) at (0,-1) {1};
\draw[\myedgestyle] (a) to node[left]{$1$}(b);

\end{tikzpicture}
\caption{$k = 1$}
\label{exist_d1}
\end{subfigure}
\hspace{2em}
\begin{subfigure}[b]{0.4\textwidth}
\centering
\begin{tikzpicture}[scale = 1.50]
\tikzset{vertex/.style = {shape=circle, draw, minimum size=1em}}
\tikzset{edge/.style = {\myedgestyle}}
%Dimension 2

\node[vertex] (a1) at (2,1) {1};
\node[vertex] (a2) at (4,1) {1};
\node[vertex] (a3) at (2,-1) {0};
\node[vertex] (a4) at (4,-1) {0};

\draw[\myedgestyle] (a1) to node[above]{$0$} (a2);
\draw[\myedgestyle] (a2) to node[right]{$-1$} (a4);
\draw[\myedgestyle] (a3) to node[below]{$0$} (a4);
\draw[\myedgestyle] (a3) to node[left]{$1$} (a1);

\end{tikzpicture}
\caption{$k = 2$}
\label{exist_d2}
\end{subfigure}
\caption{$(2,3)$-cordial $k$-dimensional oriented hypercubes}
\end{figure}
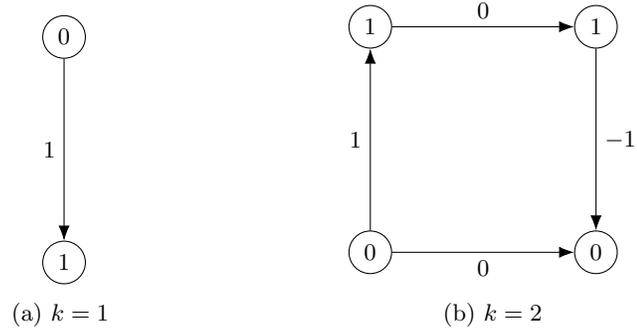

\begin{example}[Dimension $3$] Given in Figure \ref{exist_d3} is a $3$-dimensional oriented hypercube $C_3$ that is (2,3)-cordial as by the friendly vertex labeling $f$ shown, $\Lambda_{f,g}(C_3) = (4,4,4)$. 
\label{base_case}\end{example}

%Figure: Existence of Dim 3
\begin{figure}
\centering
\begin{tikzpicture}[scale = 1.35]
\tikzset{vertex/.style = {shape=circle, draw, fill = white,scale = 1}}
\tikzset{vertexpurple/.style = {shape=circle, draw,fill = purple, text = white,scale = 1}}
\tikzset{vertexorange/.style = {shape=circle, draw,fill = orange, text = white,scale = 1}}
\tikzset{edge/.style = {->,> = latex, arrows={-Stealth[scale=1.4]}}}

\newcommand{\mya}{1.3} %half side length for inner cube
\newcommand{\ax}{0.3*\mya}
\newcommand{\ay}{0.7*\mya}

\newcommand{\originx}{-0.4*\mya}
\newcommand{\originy}{-0.30*\mya}

%inner cube
\node[vertex] (a1) at (-\mya,\mya) {$\bf 1$};
\node[vertex] (a2) at (\mya,\mya) {$0$};
\node[vertex] (a3) at (-\mya,-\mya) {$0$};
\node[vertexorange] (a4) at (\mya,-\mya) {$1$};
\node[vertexpurple] (a5) at (-\mya + \ax,\mya + \ay) {$1$};
\node[vertex] (a6) at (\mya + \ax,\mya + \ay) {$0$};
\node[vertex] (a7) at (-\mya + \ax,-\mya + \ay) {$0$};
\node[vertex] (a8) at (\mya + \ax,-\mya + \ay) {$\bf 1$};

%inner cube
\draw[edge] (a1) to node[above]{$-1$}(a2);
\draw[edge] (a4) to node[below]{$-1$}(a3);
\draw[edge] (a3) to node[left]{$1$}(a1);
\draw[edge] (a2) to node[left]{$1$}(a4);
\draw[edge] (a5) to node[above]{$-1$}(a6);
\draw[edge] (a6) to node[right]{$1$}(a8);
\draw[edge] (a8) to node[below]{$-1$}(a7);
\draw[edge] (a7) to node[left]{$1$}(a5);
\draw[edge] (a1) to node[above left]{$0$}(a5);
\draw[edge] (a2) to node[left]{$0$}(a6);
\draw[edge] (a3) to node[below right]{$0$}(a7);
\draw[edge] (a4) to node[below right]{$0$}(a8);
\end{tikzpicture}
\caption{$(2,3)$-cordial 3D oriented hypercube}
\label{exist_d3}
\end{figure}
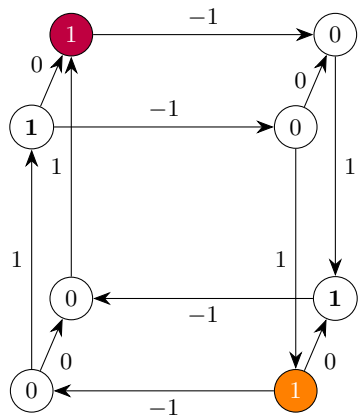

\newcommand{\lab}{h} %friendly labeling 1
\newcommand{\labb}{l} %friendly labeling 2

\newcommand{\ver}{V_h} %vertex set 1
\newcommand{\verr}{V_l} %vertex set 2

\newcommand{\bi}{\beta} %bijection type 1
\newcommand{\bii}{\delta} %bijection type 2
\newcommand{\biii}{\phi} %bijection type 3

%%%%%%%%%%%%%%%%%%%%%%%%%%%%%%%%%%
In Examples \ref{example_d1}, \ref{example_d2}, and \ref{base_case}, we see for dimension less than or equal to 3, there exist $(2,3)$-cordial oriented hypercubes. The question of existence remains unanswered for dimension greater than $3$. In the following theorem, this question is answered for the case in which dimension is a multiple of $3$.
\begin{comment}
\begin{lemma}
Let $\lab,\labb$ be friendly labelings on non-empty sets of equal size $\ver, \verr$ respectively such that $\ver,\verr$ each contain an even number of elements. Then there exists 
\begin{enumerate}
\item A bijection $\bi \colon \ver \to \verr$ such that for all $x \in \ver$, $\labb(\bi(x)) = h(x)$.
\item A bijection $\bii \colon \ver \to \verr$ such that for all $x \in \ver$, $\labb(\bii(x)) \neq h(x).$
\end{enumerate}
\end{lemma}

\begin{proof}
Given $\lab$ and $\labb$ are friendly labelings and $\ver, \verr$ are both equal in size and each contain an even amount of elements, by definition of a friendly labeling,
\begin{align}
\lvert \lab^{-1}(0) \rvert = \lvert \labb^{-1}(0) \rvert  = \lvert \lab^{-1}(1) \rvert  = \lvert \labb^{-1}(1) \rvert \label{bijection_equation}
\end{align}
\begin{enumerate}
\item By \eqref{bijection_equation}, a bijection $\bi \colon \ver\to \verr$ can be defined such that every element of $\ver$ maps to an element of $\verr$ of the same label.
\item Also by \eqref{bijection_equation}, a bijection $\bii \colon \ver \to \verr$ can be defined such that every element of $\ver$ maps to an element of $\verr$ of the opposite label.
\end{enumerate}
\end{proof}
\end{comment}
%%%%%%%%%%%%%%%%%%%%%%%%%%%%%%%%%%

\begin{theorem}
Let $n$ be a multiple of $3$, then there exists an $n$-dimensional oriented hypercube $C_n$ that is (2,3)-cordial.
\end{theorem}

\def\Lk{L_k}
\def\oLk{\overline{L_k}}
\def\Lko{L_{k+1}}
\def\oLko{\overline{L_{k+1}}}
\begin{proof}
We proceed by induction on the dimension $n$ in multiples of $3$. Example \ref{base_case} serves as a base case for $n = 3$. Suppose the claim is true for some $k$ that is a multiple of $3$. Then there exists some oriented hypercube $Q_k = (V_k,A_k)$ of dimension $k$ that is $(2,3)$-cordial. That is, there exists a friendly labeling $f \colon V_k \to \{0,1\}$ such that
$$\Lambda_{f,g} = \left(\frac{1}{3}|A_k|,\frac{1}{3}|A_k|,\frac{1}{3}|A_k|\right)$$
where $g$ is defined as in Definition \ref{def_induced_label}. We aim to construct an oriented hypercube $Q_{k+3} = (V_{k+3},A_{k+3})$ of dimension $k + 3$ that is $(2,3)$-cordial. We begin by constructing an oriented hypercube $Q_{k+1} = (V_{k+1},A_{k+1})$ of dimension $k + 1$. Let $L_k$ denote the digraph $Q_k$ with vertex labeling $f$ and induced arc labeling $g$ applied and $\oLk$ denote the digraph $Q_k$ with vertex labeling $\overline{f}$ and induced arc labeling $\overline{g}$. Now, let us draw arcs from $L_k$ to $\oLk$ according to the trivial digraph isomorphism. That is, define an arc initiating at vertex $x$ in $\Lk$ to vertex $y$ in $\oLk$ if and only if $x = y$. We then label each of these arcs as $\overline{f}(x) - f(x).$ The result is a labeled digraph, call it $L_{k+1}$. By construction, the underlying digraph of $L_{k+1}$ is an oriented hypercube of dimension $k+1$, call it $Q_{k+1}$. Define $f_{k+1}$ and $g_{k+1}$ to be vertex and arc labelings of $Q_{k+1}$ respectively such that $Q_{k+1}$ with labelings $f_{k+1}$ and $g_{k+1}$ applied is the labeled oriented hypercube $\Lko$. As $f_{k+1}$ applies friendly labelings $f$ and $\overline{f}$ to complementary subgraphs of $Q_{k+1}$, $f_{k+1}$ is a friendly labeling. Further, $g_{k+1}$ applies $g$ and $\overline{g}$ to complementary subgraphs of $Q_{k+1}$ and labels each arc $\vec{xx}$ from $\Lk$ to $\oLk$, $\overline{f}(x) - f(x)$. Then, as $f$ and $\overline{f}$ are friendly,
$$\Lambda_{f_{k+1},g_{k+1}} = \left(\frac{2}{3}|A_k| + 2^{k-1}, \frac{2}{3}|A_k| + 2^{k-1},\frac{2}{3}|A_k|\right).$$
Now, we repeat our procedure, constructing an oriented hypercube $Q_{k+2} = (V_{k+2},A_{k+2})$ of dimension $k + 2$. We draw arcs from $L_{k+1}$ and $\overline{L_{k+1}}$. Just as in the previous case, we define an arc from a vertex $x$ in $L_{k+1}$ to vertex $y$ in $\overline{L_{k+1}}$ if and only if $x = y$, and we label this arc $\overline{f_{k+1}}(x) - f_{k+1}(x)$. The result, as in the previous step, is a labeled digraph, call it $L_{k+2}$. The underlying digraph of $L_{k+2}$ is again an oriented hypercube, now of dimension $k + 2$, call it $Q_{k+2}$. As before, define $f_{k+2}$ and $g_{k+2}$ to be vertex and arc labelings of $Q_{k+2}$ respectively such that when applied to $Q_{k+2}$ yield the labeled oriented hypercube $L_{k+2}$. As before, $f_{k+2}$ applies friendly labelings $f_{k+1}$ and $\overline{f_{k+1}}$ to complementary subgraphs, thus $f_{k+2}$ is friendly. Also, $g_{k+2}$ applies $g_{k+1}$ and $\overline{g_{k+1}}$ to complementary subgraphs of $Q_{k+2}$ and labels each arc $\vec{xx}$ from $L_{k+1}$ to $\overline{L_{k+1}}$, $\overline{f_{k+1}}(x) - f_{k+1}(x)$ and $\bar{f}$. As $f_{k+1}$ and $\overline{f_{k+1}}$ are friendly,
$$\Lambda_{f_{k+2},g_{k+2}}\left(\left(\frac{4}{3}|A_k| + 2^k\right) + 2^k, \left(\frac{4}{3}|A_k| + 2^k\right) + 2^k, \frac{4}{3}|A_k|\right).$$
In our final step, we construct an oriented hypercube $Q_{k+3} = (V_{k+3},A_{k+3})$ of dimension $k + 3$ by drawing edges between two identically labeled cubes $L_{k+2}$. We draw an arc from vertex $x$ in the first $L_{k+2}$ to vertex $y$ in the second $L_{k+2}$ if and only if $x = y$ and we label this arc $f_{k+2}(x) - f_{k+2}(x) = 0$. The result is a labeled digraph, call it $L_{k+3}$. The underlying digraph of $L_{k+3}$ is an oriented hypercube of dimension $k + 3$, call it $Q_{k+3}$. Finally, we define $f_{k+3}$ and $g_{k+3}$ to be vertex and arc labelings of $Q_{k+3}$ respectively such that when applied to $Q_{k+3}$ yield the labeled oriented hypercube $L_{k+3}$. Then $f_{k+3}$ simply labels each complementary subgraph $Q_{k+2}$ according to $f_{k+2}$ and $g_{k+3}$ labels each complementary subgraph $Q_{k+2}$ according to $g_{k+2}$ and the newly drawn $2^{k+2}$ edges are labeled $0$. Let $\omega = \frac{4}{3}|A_k| + 2^{k+1}$. Then
$$\Lambda_{f_{k+3},g_{k+3}}(Q_{k+3}) = \left(2\omega, 2\omega, \frac{8}{3}|A_k| + 2^{k+2}\right).$$
Simplifying, we have
$$\Lambda_{f_{k+3},g_{k+3}}(Q_{k+3}) = \left(\frac{1}{3}(k+3)2^{k+2},\frac{1}{3}(k+3)2^{k+2},\frac{1}{3}(k+3)2^{k+2}\right).$$
As $f_{k+3}$ is constructed to be a friendly labeling, the above implies $Q_{k+3}$ is $(2,3)$-cordial.
\end{proof}

%%%%%%%%%%%%%%%%%%%%%%%%%%%%%%%%%%
\subsection{A Conjecture on Existence for Dimension $3k + 1$.}
We have now answered the question of existence of $(2,3)$-cordial oriented hypercubes for dimension less than and equal to $3$ and all dimensions which are a multiple of $3$. In this section, we now consider the existence of $(2,3)$-cordial oriented hypercubes with dimension $3k + 1$ for $k \in \N$.
%4D Example-----------------------
\begin{example}[Tesseract, Dimension $4$]
Given in Figures \ref{cordialA} and \ref{cordialB} are two 3D oriented hypercubes, $A$ and $B$, that are $(2,3)$-cordial as demonstrated by the friendly vertex labelings and induced arc labelings shown. 
\begin{figure}
\centering
\begin{subfigure}[b]{0.4\textwidth}
\centering
\begin{tikzpicture}[scale = 1.35]
\tikzset{vertex/.style = {shape=circle, draw, fill = white,scale = 1}}
\tikzset{vertexyellow/.style = {shape=circle, draw,fill = yellow, text = black,scale = 1}}
\tikzset{vertexcyan/.style = {shape=circle, draw,fill = cyan, text = black,scale = 1}}
\tikzset{edge/.style = {->,> = latex, arrows={-Stealth[scale=1.4]}}}

\newcommand{\mya}{1.3} %half side length for inner cube
\newcommand{\ax}{0.3*\mya}
\newcommand{\ay}{0.7*\mya}

\newcommand{\originx}{-0.4*\mya}
\newcommand{\originy}{-0.30*\mya}

%inner cube
\node[vertexyellow] (a1) at (-\mya,\mya) {$\bf 0$};
\node[vertex] (a2) at (\mya,\mya) {$1$};
\node[vertex] (a3) at (-\mya,-\mya) {$0$};
\node[vertex] (a4) at (\mya,-\mya) {$1$};
\node[vertex] (a5) at (-\mya + \ax,\mya + \ay) {$1$};
\node[vertex] (a6) at (\mya + \ax,\mya + \ay) {$0$};
\node[vertex] (a7) at (-\mya + \ax,-\mya + \ay) {$1$};
\node[vertexcyan] (a8) at (\mya + \ax,-\mya + \ay) {$\bf 0$};

%inner cube
\draw[edge] (a1) to node[above]{$1$}(a2);
\draw[edge] (a4) to node[below]{$-1$}(a3);
\draw[edge] (a3) to node[left]{$0$}(a1);
\draw[edge] (a2) to node[left]{$0$}(a4);
\draw[edge] (a5) to node[above]{$-1$}(a6);
\draw[edge] (a6) to node[right]{$0$}(a8);
\draw[edge] (a8) to node[below]{$1$}(a7);
\draw[edge] (a7) to node[left]{$0$}(a5);
\draw[edge] (a1) to node[above left]{$1$}(a5);
\draw[edge] (a2) to node[left]{$-1$}(a6);
\draw[edge] (a3) to node[below right]{$1$}(a7);
\draw[edge] (a4) to node[below right]{$-1$}(a8);
\end{tikzpicture}
\caption{Cube $A$}
\label{cordialA}
\end{subfigure}
\hspace{4em}
\begin{subfigure}[b]{0.4\textwidth}
\centering
\begin{tikzpicture}[scale = 1.35]
\tikzset{vertex/.style = {shape=circle, draw, fill = white,scale = 1}}
\tikzset{vertexpink/.style = {shape=circle, draw, fill = pink, text = black,scale = 1}}
\tikzset{vertexlime/.style = {shape=circle, draw, fill = lime,text = black,scale = 1}}
\tikzset{edge/.style = {->,> = latex, arrows={-Stealth[scale=1.4]}}}

\newcommand{\mya}{1.3} %half side length for inner cube
\newcommand{\ax}{0.3*\mya}
\newcommand{\ay}{0.7*\mya}

%inner cube
\node[vertex] (b1) at (-\mya,\mya) {$1$};
\node[vertexpink] (b2) at (\mya,\mya) {$\bf 0$};
\node[vertex] (b3) at (-\mya,-\mya) {$1$};
\node[vertex] (b4) at (\mya,-\mya) {$1$};
\node[vertex] (b5) at (-\mya + \ax,\mya + \ay) {$0$};
\node[vertex] (b6) at (\mya + \ax,\mya + \ay) {$1$};
\node[vertexlime] (b7) at (-\mya + \ax,-\mya + \ay) {$\bf 0$};
\node[vertex] (b8) at (\mya + \ax,-\mya + \ay) {$0$};

%outer cube
\draw[edge] (b1) to node[above]{$-1$}(b2);
\draw[edge] (b4) to node[below]{$0$}(b3);
\draw[edge] (b3) to node[left]{$0$}(b1);
\draw[edge] (b2) to node[left]{$1$}(b4);
\draw[edge] (b5) to node[above]{$1$}(b6);
\draw[edge] (b8) to node[right]{$1$}(b6);
\draw[edge] (b8) to node[below]{$0$}(b7);
\draw[edge] (b7) to node[left]{$0$}(b5);
\draw[edge] (b1) to node[above left]{$-1$}(b5);
\draw[edge] (b2) to node[left]{$1$}(b6);
\draw[edge] (b3) to node[below right]{$-1$}(b7);
\draw[edge] (b4) to node[below right]{$-1$}(b8);
\end{tikzpicture}
\caption{Cube $B$}
\label{cordialB}
\end{subfigure}
\caption{$(2,3)$-cordial 3D oriented hypercubes, $A$ and $B$}
\label{q4}
\end{figure}
In Figure \ref{q4_construction}, edges are drawn between the vertices of the oriented cube $B$ (outer) of Figure \ref{cordialB} and the vertices of oriented cube $A$ (inner) of Figure \ref{cordialA}. By the induced arc labeling scheme $g$, $2$ of these $8$ edges (red) receive an induced label of $0$ regardless of their orientation, and the remaining $6$ edges (dashed) can be oriented such that $3$ receive label $1$ and $3$ receive label $-1$, yielding a 4D oriented hypercube. As the outer and inner cubes of Figure \ref{q4_construction} have $(2,3)$-cordial labelings applied, this is to say the dashed arcs in Figure \ref{q4_construction} can be oriented such that the result is a $(2,3)$-cordial 4D oriented hypercube.
\end{example}

\begin{figure}
\centering
\begin{tikzpicture}[scale = 1.15]
\tikzset{vertex/.style = {shape=circle, draw, fill = white,scale = 0.8}}
\tikzset{vertexyellow/.style = {shape=circle, draw,fill = yellow, text = black,scale = 0.8}}
\tikzset{vertexcyan/.style = {shape=circle, draw,fill = cyan, text = black,scale = 0.8}}
\tikzset{vertexpink/.style = {shape=circle, draw, fill = pink, text = black,scale = 0.8}}
\tikzset{vertexlime/.style = {shape=circle, draw, fill = lime,text = black,scale = 0.8}}
\tikzset{edge/.style = {->,> = latex, arrows={-Stealth[scale=1.1]}}}

\newcommand{\mya}{1} %half side length for inner cube
\newcommand{\ax}{0.3*\mya}
\newcommand{\ay}{0.7*\mya}

\newcommand{\myb}{3} %half side length for outer cube
\newcommand{\bx}{0.3*\myb}
\newcommand{\by}{0.7*\myb}

\newcommand{\originx}{-0.3*\mya}
\newcommand{\originy}{-0.3*\mya}

%inner cube
\node[vertexyellow] (a1) at (-\mya,\mya) {$\bf 0$};
\node[vertex] (a2) at (\mya,\mya) {$1$};
\node[vertex] (a3) at (-\mya,-\mya) {$0$};
\node[vertex] (a4) at (\mya,-\mya) {$1$};
\node[vertex] (a5) at (-\mya + \ax,\mya + \ay) {$1$};
\node[vertex] (a6) at (\mya + \ax,\mya + \ay) {$0$};
\node[vertex] (a7) at (-\mya + \ax,-\mya + \ay) {$1$};
\node[vertexcyan] (a8) at (\mya + \ax,-\mya + \ay) {$\bf 0$};

%outer cube
\node[vertex] (b1) at (-\myb + \originx,\myb + \originy) {$1$};
\node[vertexpink] (b2) at (\myb + \originx,\myb + \originy) {$\bf 0$};
\node[vertex] (b3) at (-\myb + \originx,-\myb + \originy) {$1$};
\node[vertex] (b4) at (\myb + \originx,-\myb + \originy) {$1$};
\node[vertex] (b5) at (-\myb + \bx + \originx,\myb + \by + \originy) {$0$};
\node[vertex] (b6) at (\myb + \bx + \originx,\myb + \by + \originy) {$1$};
\node[vertexlime] (b7) at (-\myb + \bx + \originx,-\myb + \by + \originy) {$\bf 0$};
\node[vertex] (b8) at (\myb + \bx + \originx,-\myb + \by + \originy) {$0$};

%inner cube
\draw[edge] (a1) to (a2);
\draw[edge] (a4) to (a3);
\draw[edge] (a3) to (a1);
\draw[edge] (a2) to (a4);
\draw[edge] (a5) to (a6);
\draw[edge] (a6) to (a8);
\draw[edge] (a8) to (a7);
\draw[edge] (a7) to (a5);
\draw[edge] (a1) to (a5);
\draw[edge] (a2) to (a6);
\draw[edge] (a3) to (a7);
\draw[edge] (a4) to (a8);

%outer cube
\draw[edge] (b1) to (b2);
\draw[edge] (b4) to (b3);
\draw[edge] (b3) to (b1);
\draw[edge] (b2) to (b4);
\draw[edge] (b5) to (b6);
\draw[edge] (b8) to (b6);
\draw[edge] (b8) to (b7);
\draw[edge] (b7) to (b5);
\draw[edge] (b1) to (b5);
\draw[edge] (b2) to (b6);
\draw[edge] (b3) to (b7);
\draw[edge] (b4) to (b8);

%edges between cubes
\foreach \x in {1,2,3,5,6,7} {
\draw[dashed] (b\x) to node{}(a\x);
}

\foreach \x in {4,8} {
\draw[red,thick] (b\x) to node{}(a\x);
}

\end{tikzpicture}
\caption{4D $(2,3)$-Cordial Oriented Hypercube constructed from cubes $A$ and $B$}
\label{q4_construction}
\end{figure}

\begin{definition}
Let $D_1$ and $D_2$ be directed graphs with same sized vertex sets and friendly vertex labelings $f_1$ and $f_2$ respectively. Let $\beta \colon V(D_1) \to V(D_2)$ be a bijection on the vertex sets of $D_1$ and $D_2$ respectively. Then, let $h \colon V(D_1) \to \{0,1\}$ such that $h(v_1) = |f_1(v) - f_2(\beta(v))|$ for all $v \in V(D_1)$. Then define $\Phi_\beta(D_1,D_2) = |h^{-1}(0)|$. In contexts where the bijection $\beta$ is clear, we write $\Phi(D_1,D_2)$.
\end{definition}

\begin{remark}
In the context of the previous definition, given arcs are drawn between vertices of digraphs $D_1$ and $D_2$ according to the bijection $\beta$, $\Phi(D_1,D_2)$ is simply the count of arcs shared by $D_1$ and $D_2$ that receive induced label $0$ by $g$. In the following example, we work within such a context, and therefore, interpret $\Phi(D_1,D_2)$ this way.
\label{remark_phi}
\end{remark}

%7D Example-----------------------
\begin{example}[Dimension $7$]
We have introduced $3$ 3D oriented hypercubes in Figures \ref{exist_d3}, \ref{cordialA}, and \ref{cordialB} each with a $(2,3)$-cordial labeling. Let us denote the labeled oriented cube in Figure \ref{exist_d3} as $C$. For this example, we adopt the convention that $A$, $B$, and $C$ refer to labeled digraphs rather than the underlying unlabeled digraphs. We seek to construct a $(2,3)$-cordial 7D oriented hypercube from these three cubes, $A$, $B$, and $C$. As given in Figure \ref{q4_construction}, cubes $A$ and $B$ can be combined to form a 4D oriented cube such that only $2$ of the arcs they share receive label $0$. That is by the bijection between $V(A)$ and $V(B)$ defined by the edges drawn in Figure \ref{q4_construction}, $\Phi(A,B) = 2$. In Figure \ref{q7_part1}, we construct $2$ individual 4D oriented cubes, $1$ from cubes $A$ and $C$, and $1$ from cubes $B$ and $C$. As in Figure \ref{q4_construction}, arcs drawn between distinct cubes define bijections between distinct vertex sets. With respect to these bijections, in Figure \ref{q7_part1}, we see $\Phi(A,C) = \Phi(B,C) = 4$.

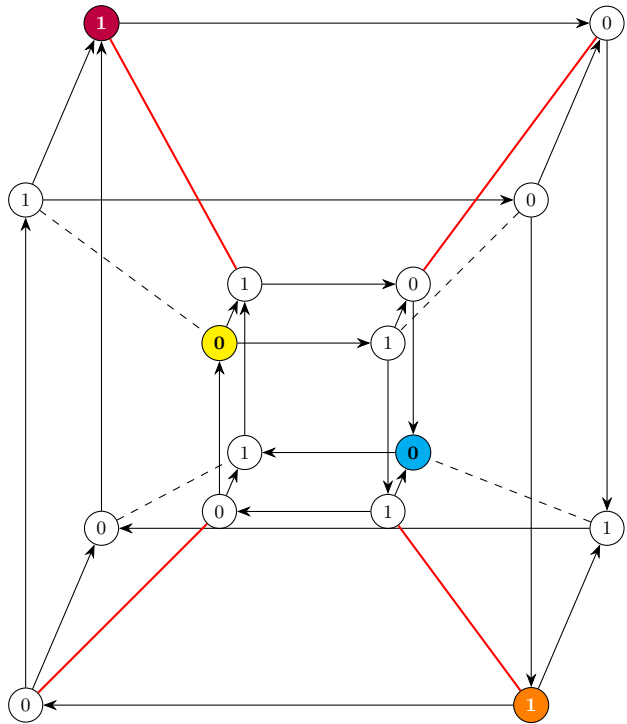
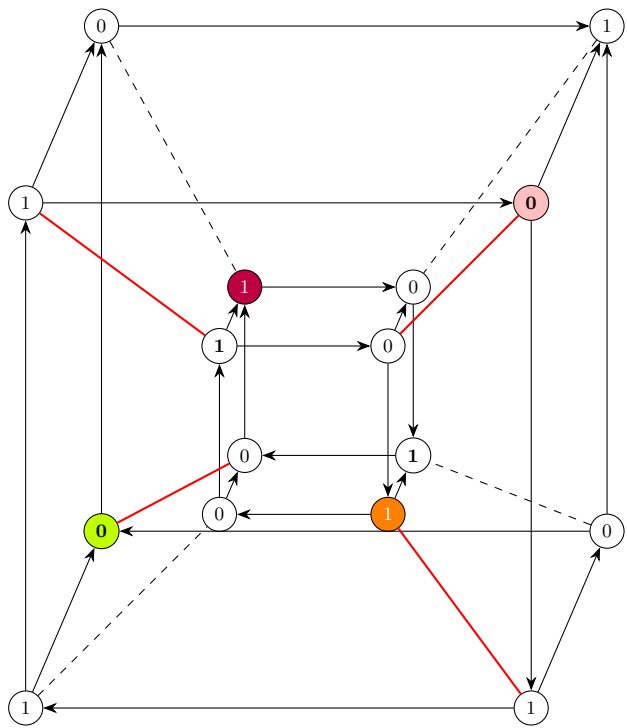
\begin{figure}
\centering
\begin{subfigure}[b]{\textwidth}
\centering
\begin{tikzpicture}[scale = 1.12]
\tikzset{vertex/.style = {shape=circle, draw, fill = white,scale = 0.8}}
\tikzset{vertexyellow/.style = {shape=circle, draw,fill = yellow, text = black,scale = 0.8}}
\tikzset{vertexcyan/.style = {shape=circle, draw,fill = cyan, text = black,scale = 0.8}}
\tikzset{vertexpurple/.style = {shape=circle, draw,fill = purple, text = white,scale = 0.8}}
\tikzset{vertexorange/.style = {shape=circle, draw,fill = orange, text = white,scale = 0.8}}
\tikzset{edge/.style = {->,> = latex, arrows={-Stealth[scale=1.1]}}}

\newcommand{\mya}{1} %half side length for inner cube
\newcommand{\ax}{0.3*\mya}
\newcommand{\ay}{0.7*\mya}

\newcommand{\myb}{3} %half side length for outer cube
\newcommand{\bx}{0.3*\myb}
\newcommand{\by}{0.7*\myb}

\newcommand{\originx}{-0.3*\mya}
\newcommand{\originy}{-0.3*\mya}

%inner cube
\node[vertexyellow] (a1) at (-\mya,\mya) {$\bf 0$};
\node[vertex] (a2) at (\mya,\mya) {$1$};
\node[vertex] (a3) at (-\mya,-\mya) {$0$};
\node[vertex] (a4) at (\mya,-\mya) {$1$};
\node[vertex] (a5) at (-\mya + \ax,\mya + \ay) {$1$};
\node[vertex] (a6) at (\mya + \ax,\mya + \ay) {$0$};
\node[vertex] (a7) at (-\mya + \ax,-\mya + \ay) {$1$};
\node[vertexcyan] (a8) at (\mya + \ax,-\mya + \ay) {$\bf 0$};

%outer cube
\node[vertex] (b1) at (-\myb + \originx,\myb + \originy) {$1$};
\node[vertex] (b2) at (\myb + \originx,\myb + \originy) {$0$};
\node[vertex] (b3) at (-\myb + \originx,-\myb + \originy) {$0$};
\node[vertexorange] (b4) at (\myb + \originx,-\myb + \originy) {$\bf 1$};
\node[vertexpurple] (b5) at (-\myb + \bx + \originx,\myb + \by + \originy) {$\bf 1$};
\node[vertex] (b6) at (\myb + \bx + \originx,\myb + \by + \originy) {$0$};
\node[vertex] (b7) at (-\myb + \bx + \originx,-\myb + \by + \originy) {$0$};
\node[vertex] (b8) at (\myb + \bx + \originx,-\myb + \by + \originy) {$1$};

%inner cube
\draw[edge] (a1) to (a2);
\draw[edge] (a4) to (a3);
\draw[edge] (a3) to (a1);
\draw[edge] (a2) to (a4);
\draw[edge] (a5) to (a6);
\draw[edge] (a6) to (a8);
\draw[edge] (a8) to (a7);
\draw[edge] (a7) to (a5);
\draw[edge] (a1) to (a5);
\draw[edge] (a2) to (a6);
\draw[edge] (a3) to (a7);
\draw[edge] (a4) to (a8);

%outer cube
\draw[edge] (b1) to (b2);
\draw[edge] (b4) to (b3);
\draw[edge] (b3) to (b1);
\draw[edge] (b2) to (b4);
\draw[edge] (b5) to (b6);
\draw[edge] (b6) to (b8);
\draw[edge] (b8) to (b7);
\draw[edge] (b7) to (b5);
\draw[edge] (b1) to (b5);
\draw[edge] (b2) to (b6);
\draw[edge] (b3) to (b7);
\draw[edge] (b4) to (b8);

%edges between cubes
\foreach \x in {1,2,7,8} {
\draw[dashed] (b\x) to node{}(a\x);
}

\foreach \x in {3,4,5,6} {
\draw[red,thick] (b\x) to node{}(a\x);
}

\end{tikzpicture}
\caption{$A$ (inner) and $C$ (outer)}
\label{q7_part1a}
\end{subfigure}

\vspace{2em}

\begin{subfigure}[b]{\textwidth}
\centering
\begin{tikzpicture}[scale = 1.12]
\tikzset{vertex/.style = {shape=circle, draw, fill = white,scale = 0.8}}
\tikzset{vertexpurple/.style = {shape=circle, draw,fill = purple, text = white,scale = 0.8}}
\tikzset{vertexorange/.style = {shape=circle, draw,fill = orange, text = white,scale = 0.8}}
\tikzset{vertexpink/.style = {shape=circle, draw, fill = pink, text = black,scale = 0.8}}
\tikzset{vertexlime/.style = {shape=circle, draw, fill = lime,text = black,scale = 0.8}}
\tikzset{edge/.style = {->,> = latex, arrows={-Stealth[scale=1.1]}}}

\newcommand{\mya}{1} %half side length for inner cube
\newcommand{\ax}{0.3*\mya}
\newcommand{\ay}{0.7*\mya}

\newcommand{\myb}{3} %half side length for outer cube
\newcommand{\bx}{0.3*\myb}
\newcommand{\by}{0.7*\myb}

\newcommand{\originx}{-0.3*\mya}
\newcommand{\originy}{-0.3*\mya}

%inner cube
\node[vertex] (a1) at (-\mya,\mya) {$\bf 1$};
\node[vertex] (a2) at (\mya,\mya) {$0$};
\node[vertex] (a3) at (-\mya,-\mya) {$0$};
\node[vertexorange] (a4) at (\mya,-\mya) {$1$};
\node[vertexpurple] (a5) at (-\mya + \ax,\mya + \ay) {$1$};
\node[vertex] (a6) at (\mya + \ax,\mya + \ay) {$0$};
\node[vertex] (a7) at (-\mya + \ax,-\mya + \ay) {$0$};
\node[vertex] (a8) at (\mya + \ax,-\mya + \ay) {$\bf 1$};

%outer cube
\node[vertex] (b1) at (-\myb + \originx,\myb + \originy) {$1$};
\node[vertexpink] (b2) at (\myb + \originx,\myb + \originy) {$\bf 0$};
\node[vertex] (b3) at (-\myb + \originx,-\myb + \originy) {$1$};
\node[vertex] (b4) at (\myb + \originx,-\myb + \originy) {$1$};
\node[vertex] (b5) at (-\myb + \bx + \originx,\myb + \by + \originy) {$0$};
\node[vertex] (b6) at (\myb + \bx + \originx,\myb + \by + \originy) {$1$};
\node[vertexlime] (b7) at (-\myb + \bx + \originx,-\myb + \by + \originy) {$\bf 0$};
\node[vertex] (b8) at (\myb + \bx + \originx,-\myb + \by + \originy) {$0$};

%inner cube
\draw[edge] (a1) to (a2);
\draw[edge] (a4) to (a3);
\draw[edge] (a3) to (a1);
\draw[edge] (a2) to (a4);
\draw[edge] (a5) to (a6);
\draw[edge] (a6) to (a8);
\draw[edge] (a8) to (a7);
\draw[edge] (a7) to (a5);
\draw[edge] (a1) to (a5);
\draw[edge] (a2) to (a6);
\draw[edge] (a3) to (a7);
\draw[edge] (a4) to (a8);

%outer cube
\draw[edge] (b1) to (b2);
\draw[edge] (b4) to (b3);
\draw[edge] (b3) to (b1);
\draw[edge] (b2) to (b4);
\draw[edge] (b5) to (b6);
\draw[edge] (b8) to (b6);
\draw[edge] (b8) to (b7);
\draw[edge] (b7) to (b5);
\draw[edge] (b1) to (b5);
\draw[edge] (b2) to (b6);
\draw[edge] (b3) to (b7);
\draw[edge] (b4) to (b8);

%edges between cubes
\foreach \x in {3,5,6,8} {
\draw[dashed] (b\x) to node{}(a\x);
}

\foreach \x in {1,2,4,7} {
\draw[red,thick] (b\x) to node{}(a\x);
}

\end{tikzpicture}
\caption{B (outer) and C (inner)}
\label{q7_part1b}
\end{subfigure}
\caption{4D oriented cubes constructed from cubes $A,B,C$}
\label{q7_part1}
\end{figure}

Now, for $D \in \{A,B,C\}$, given $f$ is the friendly vertex labeling of $D$ and $g$ is the induced arc labeling of $D$, define $\overline{D}$ to be the underlying digraph $D$ labeled instead by $\overline{f}$ and $\overline{g}$. Recall, such a labeling is $(2,3)$-cordial by Lemma 1. Then, for all $D_1,D_2 \in \{A,B,C\}$, $\Phi(\overline{D_1},\overline{D_2}) = \Phi(D_1,D_2)$ and $\Phi(\overline{D_1},D_2) = \Phi(D_1,\overline{D_2}) = 8 - \Phi(D_1,D_2)$. Then, for all $D_1 \neq D_2$, taking $\Phi(D_1,\overline{D_2})$ to be with respect to the appropriate bijection between $V(D_1)$ and $V(D_2)$ defined in either Figure \ref{q4_construction}, \ref{q7_part1a}, or \ref{q7_part1b}, we have $\Phi(\overline{A},B) = 6$ and $\Phi(\overline{B},C) = \Phi(\overline{A},C)= 4$. Lastly, note we can construct a 4D oriented hypercube between $2$ identical cubes $D$ by drawing arcs between like vertices. According to such a bijection, $\Phi(D,D) = 8$. Now, define $\gamma = \{\overline{A},A,\overline{B},B,\overline{C},C\}$. Then for all $Q_1,Q_2\in \gamma$, $\Phi(Q_1,Q_2)$ with reference to the appropriate aforementioned bijections between $V(Q_1)$ and $V(Q_2)$ are given below in Table \ref{q7_table}. As $\Phi$ is commutative by definition, the lower diagonal of Table \ref{q7_table} is left empty.

\begin{table}
\setlength{\tabcolsep}{0.5em}
\begin{center}
{\renewcommand{\arraystretch}{1.3}
\begin{tabular}{ c| c c c c c c }
  $\Phi$& $\overline{A}$ & $A$ & $\overline{B}$ & $B$ & $\overline{C}$ & $C$\\
  \hline
 $\overline{A}$ & 8 & 0 & 2 & 6 & 4 & 4 \\  
 $A$ & &8&6&2&4&4 \\
 $\overline{B}$&&&8&0&4&4\\
 $B$&&&&8&4&4\\
 $\overline{C}$&&&&&8&0\\
 $C$&&&&&&8
\end{tabular}}
\end{center}
\caption{$\Phi(Q_1,Q_2)$ for all $Q_1,Q_2 \in \gamma$}
\label{q7_table}
\end{table}

Now, we construct a $(2,3)$-cordial 7D oriented hypercube by drawing edges between cubes in the set $\gamma$ according to the previously defined vertex set bijections. Given in Figure \ref{q6from3D} are $2$ $6D$ oriented hypercubes constructed from cubes in $\gamma$ where for all $Q_1,Q_2 \in \gamma$, an edge between cube $Q_1$ and $Q_2$ signifies $8$ edges between cubes $Q_1$ and $Q_2$ drawn according to the appropriate bijection between $V(Q_1)$ and $V(Q_2)$. Note, in Figure \ref{q6from3D}, an edge between $Q_1$ and $Q_2$ is labeled $\Phi(Q_1,Q_2)$. Observe for each cube in Figure \ref{q6from3D}, the edge label sum is equal to $32$. By Remark \ref{remark_phi}, this is to say a total of $32$ edges shared by distinct cubes in $\gamma$ receive induced label $0$ by $g$ regardless of orientation. As each 6D cube in Figure \ref{q6from3D} has a total of $12 \cdot 8 = 96$ edges drawn between cubes in $\gamma$, and $32 = 96/3$, the remaining edges not labeled $0$ in each 6D cube can be oriented such that half are labeled $1$ and half are labeled $-1$ making each 6D cube $(2,3)$-cordial. Now, in Figure \ref{q7_part3} a 7D cube is constructed from these $(2,3)$-cordial 6D oriented cubes. In Figure \ref{q7_part3} as in Figure \ref{q6from3D}, an edge between cube $Q_1$ and $Q_2$ signifies $8$ edges between cubes $Q_1$ and $Q_2$ drawn according to the appropriate bijection between $V(Q_1)$ and $V(Q_2)$, and each edge between $Q_1$ and $Q_2$ is labeled $\Phi(Q_1,Q_2)$. 

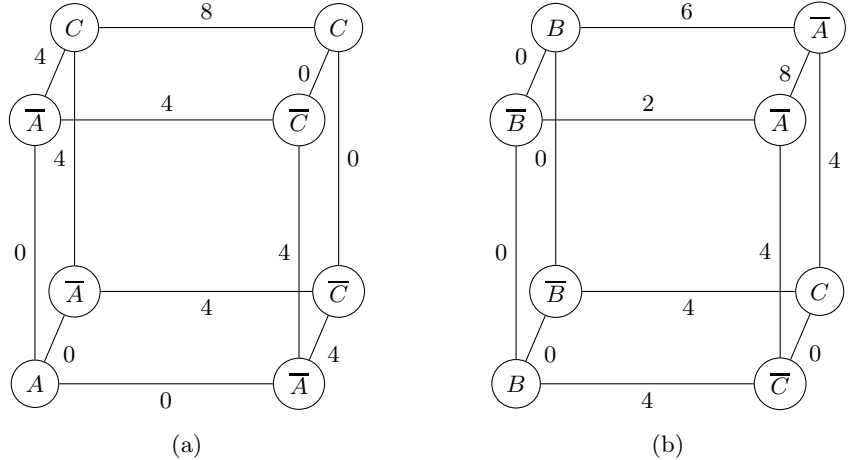
\begin{figure}
\centering
\begin{subfigure}[b]{0.4\textwidth}
\centering
\begin{tikzpicture}[scale = 1.35]
\tikzset{vertex/.style = {shape=circle, draw, fill = white,scale = 1}}
\tikzset{vertexyellow/.style = {shape=circle, draw,fill = yellow, text = black,scale = 1}}
\tikzset{vertexcyan/.style = {shape=circle, draw,fill = cyan, text = black,scale = 1}}
\tikzset{edge/.style = {->,> = latex, arrows={-Stealth[scale=1.4]}}}

\newcommand{\mya}{1.3} %half side length for inner cube
\newcommand{\ax}{0.3*\mya}
\newcommand{\ay}{0.7*\mya}

\newcommand{\originx}{-0.4*\mya}
\newcommand{\originy}{-0.30*\mya}

%inner cube
\node[vertex] (a1) at (-\mya,\mya) {$\overline{A}$};
\node[vertex] (a2) at (\mya,\mya) {$\overline{C}$};
\node[vertex] (a3) at (-\mya,-\mya) {$A$};
\node[vertex] (a4) at (\mya,-\mya) {$\overline{A}$};
\node[vertex] (a5) at (-\mya + \ax,\mya + \ay) {$C$};
\node[vertex] (a6) at (\mya + \ax,\mya + \ay) {$C$};
\node[vertex] (a7) at (-\mya + \ax,-\mya + \ay) {$\overline{A}$};
\node[vertex] (a8) at (\mya + \ax,-\mya + \ay) {$ \overline{C}$};

%inner cube
\draw[] (a1) to node[above]{$4$}(a2);
\draw[] (a4) to node[below]{$0$}(a3);
\draw[] (a3) to node[left]{$0$}(a1);
\draw[] (a2) to node[left]{$4$}(a4);
\draw[] (a5) to node[above]{$8$}(a6);
\draw[] (a6) to node[right]{$0$}(a8);
\draw[] (a8) to node[below]{$4$}(a7);
\draw[] (a7) to node[left]{$4$}(a5);
\draw[] (a1) to node[above left]{$4$}(a5);
\draw[] (a2) to node[left]{$0$}(a6);
\draw[] (a3) to node[below right]{$0$}(a7);
\draw[] (a4) to node[below right]{$4$}(a8);
\end{tikzpicture}
\caption{}
\end{subfigure}
\hspace{4em}
\begin{subfigure}[b]{0.4\textwidth}
\centering
\begin{tikzpicture}[scale = 1.35]
\tikzset{vertex/.style = {shape=circle, draw, fill = white,scale = 1}}
\tikzset{vertexpink/.style = {shape=circle, draw, fill = pink, text = black,scale = 1}}
\tikzset{vertexlime/.style = {shape=circle, draw, fill = lime,text = black,scale = 1}}
\tikzset{edge/.style = {->,> = latex, arrows={-Stealth[scale=1.4]}}}

\newcommand{\mya}{1.3} %half side length for inner cube
\newcommand{\ax}{0.3*\mya}
\newcommand{\ay}{0.7*\mya}

%inner cube
\node[vertex] (b1) at (-\mya,\mya) {$\overline{B}$};
\node[vertex] (b2) at (\mya,\mya) {$\overline{A}$};
\node[vertex] (b3) at (-\mya,-\mya) {$B$};
\node[vertex] (b4) at (\mya,-\mya) {$\overline{C}$};
\node[vertex] (b5) at (-\mya + \ax,\mya + \ay) {$B$};
\node[vertex] (b6) at (\mya + \ax,\mya + \ay) {$\overline{A}$};
\node[vertex] (b7) at (-\mya + \ax,-\mya + \ay) {$\overline{B}$};
\node[vertex] (b8) at (\mya + \ax,-\mya + \ay) {$C$};

%outer cube
\draw[] (b1) to node[above]{$2$}(b2);
\draw[] (b4) to node[below]{$4$}(b3);
\draw[] (b3) to node[left]{$0$}(b1);
\draw[] (b2) to node[left]{$4$}(b4);
\draw[] (b5) to node[above]{$6$}(b6);
\draw[] (b8) to node[right]{$4$}(b6);
\draw[] (b8) to node[below]{$4$}(b7);
\draw[] (b7) to node[left]{$0$}(b5);
\draw[] (b1) to node[above left]{$0$}(b5);
\draw[] (b2) to node[left]{$8$}(b6);
\draw[] (b3) to node[below right]{$0$}(b7);
\draw[] (b4) to node[below right]{$0$}(b8);
\end{tikzpicture}
\caption{}
\end{subfigure}
\caption{$(2,3)$-cordial 6D oriented hypercubes}
\label{q6from3D}
\end{figure}

In Figure \ref{q7_part3}, the edge label sum is $22$. Similar to before, by Remark \ref{remark_phi}, this is to say a total of $22$ of the $64$ edges drawn between vertices of the inner 6D cube and the outer 6D cube receive label $0$. The remaining $42$ edges can be oriented such that $21$ receive a label of $1$ and $21$ receive a label of $-1$ by $g$. Because each 6D cube is $(2,3)$-cordial, such a choice yields a $(2,3)$-cordial 7D oriented hypercube.
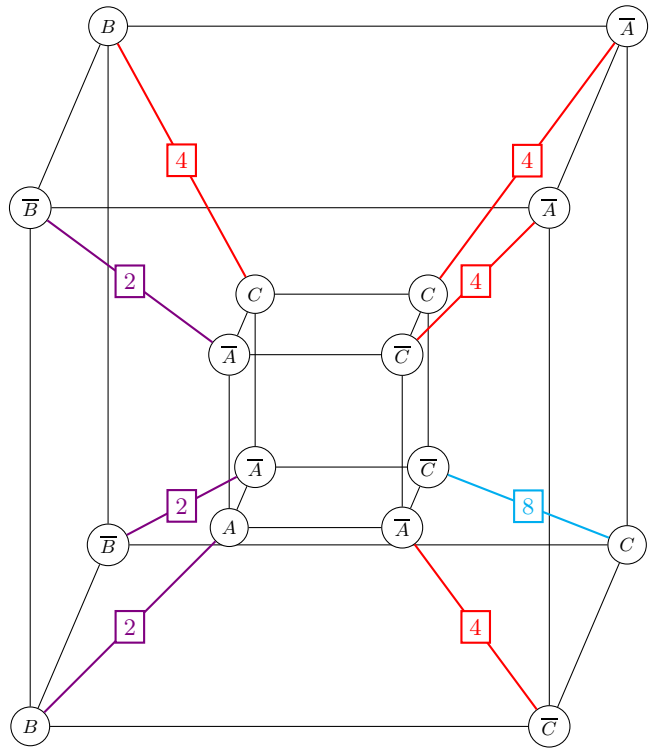
\begin{figure}
\centering
\begin{tikzpicture}[scale = 1.15]
\tikzset{vertex/.style = {shape=circle, draw, fill = white,scale = 0.8}}
\tikzset{vertexyellow/.style = {shape=circle, draw,fill = yellow, text = black,scale = 0.8}}
\tikzset{vertexcyan/.style = {shape=circle, draw,fill = cyan, text = black,scale = 0.8}}
\tikzset{vertexpink/.style = {shape=circle, draw, fill = pink, text = black,scale = 0.8}}
\tikzset{vertexlime/.style = {shape=circle, draw, fill = lime,text = black,scale = 0.8}}
\tikzset{edge/.style = {->,> = latex, arrows={-Stealth[scale=1.1]}}}

\newcommand{\mya}{1} %half side length for inner cube
\newcommand{\ax}{0.3*\mya}
\newcommand{\ay}{0.7*\mya}

\newcommand{\myb}{3} %half side length for outer cube
\newcommand{\bx}{0.3*\myb}
\newcommand{\by}{0.7*\myb}

\newcommand{\originx}{-0.3*\mya}
\newcommand{\originy}{-0.3*\mya}

%outer cube
\node[vertex] (b1) at (-\myb + \originx,\myb + \originy) {$\overline{B}$};
\node[vertex] (b2) at (\myb + \originx,\myb + \originy) {$\overline{A}$};
\node[vertex] (b3) at (-\myb + \originx,-\myb + \originy) {$B$};
\node[vertex] (b4) at (\myb + \originx,-\myb + \originy) {$\overline{C}$};
\node[vertex] (b5) at (-\myb + \bx + \originx,\myb + \by + \originy) {$B$};
\node[vertex] (b6) at (\myb + \bx + \originx,\myb + \by + \originy) {$\overline{A}$};
\node[vertex] (b7) at (-\myb + \bx + \originx,-\myb + \by + \originy) {$\overline{B}$};
\node[vertex] (b8) at (\myb + \bx + \originx,-\myb + \by + \originy) {$C$};

%outer cube
\draw[] (b1) to (b2);
\draw[] (b4) to (b3);
\draw[] (b3) to (b1);
\draw[] (b2) to (b4);
\draw[] (b5) to (b6);
\draw[] (b8) to (b6);
\draw[] (b8) to (b7);
\draw[] (b7) to (b5);
\draw[] (b1) to (b5);
\draw[] (b2) to (b6);
\draw[] (b3) to (b7);
\draw[] (b4) to (b8);

%inner cube
\node[vertex] (a1) at (-\mya,\mya) {$\overline{A}$};
\node[vertex] (a2) at (\mya,\mya) {$\overline{C}$};
\node[vertex] (a3) at (-\mya,-\mya) {$A$};
\node[vertex] (a4) at (\mya,-\mya) {$\overline{A}$};
\node[vertex] (a5) at (-\mya + \ax,\mya + \ay) {$C$};
\node[vertex] (a6) at (\mya + \ax,\mya + \ay) {$C$};
\node[vertex] (a7) at (-\mya + \ax,-\mya + \ay) {$\overline{A}$};
\node[vertex] (a8) at (\mya + \ax,-\mya + \ay) {$\overline{C}$};

%inner cube
\draw[] (a1) to (a2);
\draw[] (a4) to (a3);
\draw[] (a3) to (a1);
\draw[] (a2) to (a4);
\draw[] (a5) to (a6);
\draw[] (a6) to (a8);
\draw[] (a8) to (a7);
\draw[] (a7) to (a5);
\draw[] (a1) to (a5);
\draw[] (a2) to (a6);
\draw[] (a3) to (a7);
\draw[] (a4) to (a8);

%edges between cubes
\foreach \x in {1,3,7} {
\draw[violet,thick] (b\x) to node[draw,fill = white]{$2$}(a\x);
}

\foreach \x in {2,4,5,6} {
\draw[red,thick] (b\x) to node[draw,fill = white]{$4$}(a\x);
}

\foreach \x in {8} {
\draw[cyan,thick] (b\x) to node[draw,fill = white]{$8$}(a\x);
}

\end{tikzpicture}
\caption{7D $(2,3)$-cordial oriented hypercube constructed from $\gamma$ cubes}
\label{q7_part3}
\end{figure}
\end{example}
In the previous two examples we have confirmed there exist $(2,3)$-cordial oriented hypercubes of dimension $3k + 1$ for $k = 1,2$. We now state the following conjecture.
\begin{conjecture}
Let $n$ be a multiple of $3$, then there exists an $(n + 1)$-dimensional oriented hypercube $C_{n+1}$ that is (2,3)-cordial.
\end{conjecture}

\section{Non-$(2,3)$-Cordial Oriented Cubes.}
In the previous section, we demonstrated the existence of $(2,3)$-cordial oriented hypercubes of varying dimension including dimension $3$. Now, we demonstrate the existence of oriented cubes that are not $(2,3)$-cordial, that is, we demonstrate there exist oriented cubes that do not admit $(2,3)$-cordial labelings.

\begin{figure}
\centering
\begin{tikzpicture}[scale = 1.35]
\tikzset{vertex/.style = {shape=circle, draw, fill = white,scale = 1}}
\tikzset{vertexred/.style = {shape=circle, draw,fill = pink, text = black,scale = 1}}
\tikzset{vertexblue/.style = {shape=circle, draw,fill = cyan, text = black,scale = 1}}
\tikzset{edge/.style = {->,> = latex, arrows={-Stealth[scale=1.4]}}}

\newcommand{\mya}{1.3} %half side length for inner cube
\newcommand{\ax}{0.3*\mya}
\newcommand{\ay}{0.7*\mya}

\newcommand{\originx}{-0.4*\mya}
\newcommand{\originy}{-0.30*\mya}

%inner cube
\node[vertex] (a1) at (-\mya,\mya) {$s_1$};
\node[vertexred] (a2) at (\mya,\mya) {$b_3$};
\node[vertex] (a3) at (-\mya,-\mya) {$s_3$};
\node[vertexblue] (a4) at (\mya,-\mya) {$v_1$};
\node[vertexred] (a5) at (-\mya + \ax,\mya + \ay) {$b_1$};
\node[vertexblue] (a6) at (\mya + \ax,\mya + \ay) {$v_2$};
\node[vertex] (a7) at (-\mya + \ax,-\mya + \ay) {$s_2$};
\node[vertexred] (a8) at (\mya + \ax,-\mya + \ay) {$b_2$};

%inner cube
\draw[edge] (a2) to node[above]{}(a1);
\draw[edge] (a3) to node[below]{}(a4);
\draw[edge] (a1) to node[left]{}(a3);
\draw[edge] (a2) to node[left]{}(a4);
\draw[edge] (a5) to node[above]{}(a6);
\draw[edge] (a8) to node[right]{}(a6);
\draw[edge] (a8) to node[below]{}(a7);
\draw[edge] (a5) to node[left]{}(a7);
\draw[edge] (a5) to node[above left]{}(a1);
\draw[edge] (a2) to node[left]{}(a6);
\draw[edge] (a7) to node[below right]{}(a3);
\draw[edge] (a8) to node[below right]{}(a4);
\end{tikzpicture}
\caption{Oriented cube $V$, $3$ vertices of out-degree $3$ (labeled $b_i$), and $2$ vertices of in-degree $3$ (labeled $v_i$)}
\label{non_cord_1}
\end{figure}
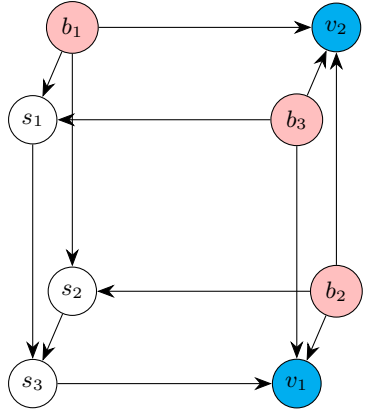
\begin{theorem}
The oriented cube $V$ in Figure \ref{non_cord_1} is not $(2,3)$-cordial.
\label{non_cord}
\end{theorem}
\begin{proof}
There are \(\binom{8}{4}\) possible friendly vertex labelings for the oriented cube $V$. By a brute force algorithm, it can be shown that none of these vertex labelings induces a $(2,3)$-cordial labeling.
\end{proof}

\begin{corollary}
The oriented cube $V^R$ for $V$ in Figure \ref{non_cord_1} is not $(2,3)$-cordial.
\end{corollary}
\begin{proof}
By Lemma \ref{lambda_lemma}.1, $\Lambda_{f,g}(V) = \Lambda_{f,g}(V^R)$ for any vertex-arc labeling $f,g$. Thus, given $V$ does not admit a $(2,3)$-cordial labeling by Theorem \ref{non_cord}, neither does $V^R$. Equivalently, $V^R$ is not $(2,3)$-cordial.
\end{proof}

\begin{theorem}
The cubes $V$ and $V^R$ are the only oriented cubes up to isomorphism that are not $(2,3)$-cordial.
\end{theorem}
\begin{proof}
This can be shown by a simple brute force algorithm.
\end{proof}

 \end{document}